\documentclass[12pt, reqno]{amsproc}
\usepackage{amssymb,amsmath,amsfonts,amssymb, latexsym, verbatim, esint,amsthm, mathrsfs}
\usepackage[numbers,sort&compress]{natbib}
\usepackage{color}
\newcommand{\be}{\begin{equation}}
\newcommand{\ee}{\end{equation}}
\newcommand{\la}{\label}
\newcommand{\ba}{\begin{array}{l}}
\newcommand{\ea}{\end{array}}
\newcommand{\Rr}{{\mathbb R}}
\newcommand{\nax}{\nabla_x}
\newcommand{\pa}{\partial}
\newcommand{\fr}{\frac}
\newcommand{\na}{\nabla}
\newtheorem{thm}{Theorem}

\newtheorem{lemma}{Lemma}
\theoremstyle{definition}

\newtheorem{rem}{Remark}
\newcommand{\beg}{\begin}

\renewcommand{\div}{\mathrm{div}}
\newcommand{\p}{\lbrack p\rbrack} 

\title{Radiative Vlasov-Maxwell Equations}

\author{Peter Constantin}
\address{Department of Mathematics, Princeton University, Princeton, NJ 08544}
\email{const@math.princeton.edu}

\author{Hezekiah Grayer II}
\address{Program in Applied and Computational Mathematics, Princeton University, Princeton, NJ 08544}
\email{hgrayer@math.princeton.edu}

\date{today}
\begin{document}
\begin{abstract}
The {{R}}adiative Vlasov-Maxwell equations {{model the radiative kinetics of collisionless relativistic plasma. In them}} the Lorentz force is modified by the addition of  radiation reaction forces. The radiation forces produce damping of particle energy but these forces are not divergence-free in momentum space, which has an effect of concentration near zero momentum. We prove unconditional global regularity of solutions for a class of {R}adiative Vlasov-Maxwell equations with large initial data. 
\end{abstract}

        
        \noindent

\maketitle
%version 6/18/25
\section{Introduction}

Radiation reaction forces in plasma capture the irreversible transfer of kinetic energy into radiation as the charged particles accelerate. There are several models of this phenomenon in the physical literature \cite{landau1975classical}, \cite{rybicki1991radiative} and a formal derivation from microscopic models \cite{kuz1978bogolyubov}. These models apply to relativistic plasma often found in high-energy astrophysical systems.  Still, a rigorous self-consistent derivation of the particle dynamics and their radiation is fraught with fundamental
challenges \cite{spohn2004dynamics}.  Radiative forces are significant for particles at large velocities and are not accounted for in the classical Vlasov-Maxwell equations.  In this paper we prove the global regularity of solutions with large initial data for a class of Radiative Vlasov-Maxwell equations. We are not aware of any mathematical analysis of the RVM equations.

In contrast to RVM, the problem of global regularity for solutions of the classical Vlasov-Maxwell equations with large data has been studied extensively, but remains unsolved. The Vlasov-Maxwell equations are locally well posed \cite{asano1986local}. Small data results have been obtained  \cite{glassey1987absence}, \cite{schaeffer2004small},  in which the plasma is initially dilute, the solutions remain small and smooth, disperse and their asymptotic behavior is free (\cite{bigorgne2020sharp,bigorgne2021asymptotic,bigorgne2022global,bigorgne2023scattering}).  This picture holds for nearly neutral data as well  (\cite{glassey1988global},\cite{breton2025modified}). 
There are several recent results (\cite{han2024linear,nguyen2024new}) concerning the asymptotic behavior of small  perturbations of steady states which do not depend on the space variable. %, in the 3D homogeneous setting.
Existence of  global weak solutions was obtained in \cite{diperna1989global}. 

For smooth large data, the possibility of spontaneous singularity formation has been the focus of many analytical works.  In seminal papers,  Glassey and Strauss \cite{glassey1986singularity, glassey1989large}  proved that the only way singularities might arise in finite time is through  concentration of particle density at very high velocity. Specifically, they proved that if the solution-averaged Lorentz factor  $\langle \gamma \rangle$ is  uniformly bounded, then no singularities can form in finite time from smooth and localized initial data. The quantity $\langle \gamma \rangle$ is a function of space and time representing the kinetic energy density of the particles.
In \cite{klainerman2002new} it was shown using Fourier analysis that the singularities are averted if the electromagnetic fields remain bounded. Several other results are based on Fourier methods \cite{bouchut2003classical,bouchut2004nonresonant},\cite{pallard2005non}.

A number of extensions of the results of Glassey and Strauss concern moments of the type 
$M_{\theta,q} = \| \langle \gamma ^\theta \rangle \|_{L^{q}(dx)}$. In our notation, 
$$\langle \gamma^\theta \rangle = \int_{\mathbb{R}^3} (\sqrt{1 + |p|^2})^\theta f(x,p,t) \, dp$$ for an exponent $\theta$. The average  of the kinetic energy density considered by Glassey and Strauss corresponds to  $M_{1,\infty}$.
 In \cite{pallard2005boundedness}, control of $M_{\theta, q}$ where $\theta > 4/q$ and $6 \leq q \leq \infty$ is shown to be sufficient for regularity. In  \cite{sospedra2010classical}, control  of $M_{0,\infty}$ is established as a regularity criterion, and in \cite{pallard2015refined} this result was extended to $M_{0,6}$. The results of
\cite{kunze2015yet} imply that finiteness of $M_{3,2}$ is sufficient for regularity. In \cite{luk2016strichartz} it is shown that for regularity, if $2 < q \leq \infty$ and $\theta > 2/q$, then control of $M_{\theta,q}$ is sufficient, and if $1 \leq q \leq 2$ and $\theta > 8/q - 3$, then control of $M_{\theta,q}$ is sufficient, and an improvement \cite{patel2018three} shows that if $\theta > 3$, then control of  $M_{\theta,1}$ is sufficient for regularity.  In \cite{luk2014new},  it is proven that the solutions remain smooth if a plane projection of the momenta is bounded through the evolution. Results  of global regularity for cylindrical symmetry are announced in \cite{wang2022global}.

The Vlasov-Maxwell (VM) equations are formed by the Vlasov equation for the particle distribution function $f= f(x,p,t)$, coupled to the Maxwell equations for the electromagnetic (EM) fields $E = E(x,t)$ and $B = B(x,t)$. The particle dynamics is driven by the Lorentz force $$F_L  = E + v\times B.$$  The Radiative Vlasov-Maxwell (RVM) equations are the same equations, except that the particles are moved by a total force $$F= F_L + F_R$$ where $F_R$ is the radiation reaction force.  The RVM equations are not a small perturbation of the classical Vlasov-Maxwell equations. 
The main result of this paper is:
\beg{thm} \la{main}Assume that the initial data $E_0(x)$ and $B_0(x)$ for the electromagnetic fields $E(x,t)$ and $B(x,t)$ and the initial data $f_0(x,p)$ for the particle distribution function $f(x,p,t)$ are smooth, compatible, and decay at spatial infinity. In addition assume 
\[
f_0(x,0) =0
\]
(the initial particle distribution vanishes at zero momentum) and
\[
\sup_{x,p} f_0(x,p) \exp{(A_0|p|)} < \infty
\]
holds for some  $A_0> 0$ large enough  (the initial particle density decays uniformly exponentially at high momentum).
 Then, the solution of the RVM equations is globally smooth and
there exist constants $C$ depending explicitly only on the initial data so that
\[
\begin{aligned}
 |E(x,t)| + |B(x,t)| + |\na_x E(x,t)| + |\na_x &B(x,t)| \le C\exp(Ct)
\end{aligned}
\]
and
\[
f(x,p,t) + |\na_x f(x,p,t)| + \sqrt{1 + |p|^2} |\na_p f(x,p,t)| \le C\exp( C\exp (Ct))
\]
hold for all $x, p$ and $t$.
\end{thm}

In this paper we address the main problem, which is to obtain unconditional global a priori bounds for general large data. We do not strive for the most economical function spaces, nor provide a construction of solutions. The construction of solutions, asymptotic behavior for small data and analysis of related models will be discussed in forthcoming works. We chose for simplicity the single species model, but the same proof applies to the multispecies model. We also chose to emphasize unconditional results, based on precisely specified reaction forces. Physically motivated conditional results for more general forces, assuming bounds on the EM fields may also be obtained with our methods.

Some ideas of the proof and a comparison with the VM equations are given below. Unlike the VM equations, where the total force $F_L$ is divergence-free in $p$, $\div_p F_L=0$, the radiative force's divergence
\[
\div_p F_R \neq 0
\]
is mostly negative. Thus, unlike the VM case where $f$ is automatically bounded if initially so, in the RVM equations $f$ is not bounded uniformly and can (and will) grow in time. The danger is uncontrolled implosion, because the phase volume is contracting. On the other hand, the radiation reaction force causes the flux of the kinetic energy density to decay. Thus, the main danger of singularity formation in RVM, as opposed to VM, comes not from high, but from low velocity.

The radiation reaction force is used to obtain unconditional a priori bounds on the particle distribution, which blow up like $|p|^{-3}$ near the origin, but decay exponentially at large $|p|$. This is a manifestation of the damping at high momenta, and the price one pays for the negative divergence of forces. The singular bounds on the particle distribution function make it impossible to bound directly the charge density, but they imply unconditional a priori bounds 
\[
\langle |v| \lbrack p \rbrack^n\rangle (x,t) = \int_{\Rr^3} |v| \lbrack p \rbrack^n f(x,p,t)dp \le M_n
\]
(in our notation the Lorentz factor is $\gamma = \lbrack p \rbrack = \sqrt{1 + |p|^2}$, with the normalized speed of light $c=1$, the velocity is $v = p / \lbrack p \rbrack$ and $p$ is the momentum). 
These ``fluxes of moments'' bounds are not in by themselves bounds on the moments because $v$ vanishes at $p=0$, but in the next step  we deduce new ``flux of energy''-type bounds in terms of  fluxes of moments and logarithms of gradients of $f$. Here we have to use the propagation of the condition $f(x,0,t)=0$ due to the annihilation of the contribution of the electric field at zero momentum. This is the reason the charge density will turn out to be finite, albeit growing at a double exponential rate in time.
To close the bounds we now turn to the Glassey-Strauss method of representing the electromagnetic fields. Using it and the gradient-conditioned moment bounds, we obtain bounds on the EM fields in terms of a choice of logarithms of gradients of $f$, in other words, in terms of a quantity \\
 $$\min{\left \{\log_+\|\na_p f(t)\|_{L^{\infty}}, \sup_{s\le t}\log_+  \|\na_x f(s)\|_{L^{\infty}}\right\}}.$$ 
  The Glassey-Strauss representation for gradients is then used together with the EM bounds to obtain a priori estimates of the gradients of the EM fields. Finally, we apply the bounds on the EM fields and their gradients to bound the gradients of $f$, closing the argument. Ultimately, global regularity is a consequence of superlinear differential inequalities for the gradients of $f$, with doubly logarithmic nonlinearity.  The EM field bound in terms of minimum of two gradient logarithms is crucial in order to obtain global regularity: without this minimum, our bounds would not be sufficient to rule out finite time blow up.

The paper is organized as follows: After a section on notation and preliminaries (Section \ref{prel}) where we describe the RVM equations, we make specific the form of the radiation reaction force $F_R$ and  summarize its properties in Section \ref{propf}. We recall the Glassey-Strauss representation in Section \ref{gs}, and in Section \ref{vm} we derive moment bounds. In Section \ref{embounds} we obtain bounds on the EM fields and in Section \ref{gembounds} we derive bounds for their gradients. In Section \ref{gfb} we obtain the final gradient bounds on $f$ and conclude the proof of Theorem \ref{main}. In Appendix A we verify some properties of the Glassey-Strauss representation and in Appendix B we give the proofs of ODE lemmas.

\section{Preliminaries: notation, the RVM equations}\la{prel}
The radiative Vlasov-Maxwell  equations are formed with the Vlasov equation
\begin{equation}
\pa_t f + \div_x(v f) + \div_p(F f)=0,
\la{rmv}
\end{equation}
with
$f(x,p,t)\ge 0$, $(x,p,t) \in \Rr^3\times\Rr^3\times \Rr$ and 
\begin{equation}
F =  F_L + F_R 
\la{F}
\end{equation}
where $F_L$ is the Lorentz force
\begin{equation}
F_L = E + v\times B
\la{FL}
\end{equation} 
and $F_R$ is the radiation reaction force, which will be discussed in the next section (see Definition \ref{simple}).
The velocity is denoted by $v$,
\begin{equation}
v = \fr{p}{\sqrt{1+ |p|^2}} = \fr{p}{\lbrack p \rbrack},
\la{vp}
\end{equation}
and the Lorentz factor $\gamma$ by $[p]$, 
\begin{equation}
\p = \sqrt{1 + |p|^2}.
\la{pnot}
\end{equation}
$E(x,t)$ and $B(x,t)$ are respectively the electric field and the magnetic field.
They solve the Maxwell equations,
\begin{equation}
\left\{\,\,
\begin{aligned}
\pa_t E - \na_x\times B &=  - j,  \\
\div_x E &= \rho  \\ 
\pa_t B + \na_x\times E&= 0 \\
\div_x B &= 0,
\la{Maxwell}
\end{aligned}
\right.
\end{equation}
together with
\begin{equation}
\rho = \int f dp  = \langle 1\rangle \quad \text{and} \quad j = \int v fdp = \langle v \rangle.
\la{rhoj}
\end{equation}
Throughout the paper, for a function $\phi(x,p,t)$, we denote the solution average
\be
\langle \phi\rangle(x,t) = \int \phi(x,p,t)f(x,p,t)dp.
\la{langlerangle}
\ee
The RVM equations are comprised of \eqref{rmv} with \eqref{F} and \eqref{Maxwell} with \eqref{rhoj}. Smooth solutions of RVM require the following compatibility conditions to be satisfied by the initial data: $f_0\geq 0$, 
\begin{equation}
\div_x E_0 = \int f_0 \, dp \quad \text{ and }\quad \div_x B_0 = 0.
\end{equation}

\section{The radiation reaction force}\la{propf}
We write 
\begin{equation}
{\mathbf {K}}(x,t) = (E(x,t), B(x,t)) 
\la{KEB}
\end{equation}
and
\begin{equation}
K^2 =  |E|^2 + |B|^2 =  |{\mathbf{K}}|^2.
\la{[K]}
\end{equation}
\beg{defi} \la{simple}
In this paper, the \emph{radiation reaction force} is
\begin{equation*}
F_R(x,p,t) = - \chi(|p|) E(x,p,t) - Mp K(x,t) 
\end{equation*}
with $M>2$ a constant. Here $0\le \chi\le 1$ is a smooth cutoff,
\begin{equation*}
\chi (r) = 1 \;\text{for} \; r\le R_0 \;\text{and} \; \chi(r) = 0 \;\text{for}\; r\ge R_1,  |\chi' (r)|\le 2.
\la{chir}
\end{equation*}
\end{defi}
\beg{rem}\la{forces}
Some of the examples of radiation reaction forces in the physical literature include  (\cite{landau1975classical})
\[
F_{LL} = -h v \gamma^2 (|F_L|^2 - (v \cdot E)^2)
\]
and the force due to inverse Compton scattering (\cite{rybicki1991radiative})
\[
F_{IC} = -h v \gamma^2 K^2.
\]
The parameter $h > 0$ measures the relative intensity of the reaction, and is proportional to Planck's constant.  These examples grow quadratically with the EM fields and vanish at $p=0$. 
In the present work we use the  term $-\chi E$ to mitigate the effect of the electric field at $p=0$, and the linear growth of $F_R$ in the EM fields to close an a priori bound on the EM fields using a bootstrap argument. The form in Definition \ref{simple} was chosen for its simplicity, many other similar expressions, including modifications of $F_{LL}$ and $F_{IC}$  will provide the same unconditional result. Because the unmodified expressions $F_{LL}$ and $F_{IC}$  grow quadratically with the size of the EM fields, in these cases our methods provide conditional global regularity for large data, assuming that the EM fields are bounded.
\end{rem}

The  effect of the radiation reaction force as it pertains to regularity is as follows.
Writing $\widehat p = p / |p|$,  we find
\begin{equation}
\begin{aligned}
F\cdot \widehat p &= (1-\chi(|p|))E\cdot \widehat p - MK |p| \\
&\leq -K(x,t)(M|p| -(1-\chi(|p|))) \leq 0
\end{aligned}
\la{Fpineq}
\end{equation}
holds  because 
\begin{equation}
M\ge 2, \quad 1-\chi(r) \le 2r.
\la{MR1}
\end{equation}
We note that
\begin{equation}
\div_p{F_L} = 0,
\la{divpFL}
\end{equation}
however,  $\div_p F_R\neq 0$; in fact
\begin{equation}
-\div_p F = 3MK(x,t) + \chi'(|p|) E\cdot\widehat p.
\la{divpFsimple}
\end{equation}
We show in Section \ref{vm} that for large enough positive constants $A$,
\begin{equation}
\left(\fr{3}{|p|} + A\right)F\cdot\widehat p - \div_p F \le 0
\la{condA}
\end{equation}
holds. This is a key property of $F$.

Observe that
\begin{equation}
|F(x,p,t)| \le (M+2) |p| K(x,t) ,
\la{Fbound}
\end{equation}
and differentiating, we find
\begin{equation}
|\na_p F(x,p,t)|+ |\na_p \na_p F(x,p,t)| \le C(M+2) K(x,t).
\la{napFbound}
\end{equation}
Moreover,
\begin{equation}
|\na _x F(x,p,t)| \le C |p| (|\na_x E| + |\na_x B| + K(x,t)),
\la{naxFbound}
\end{equation}
and 
\begin{equation}
|\na_p\nax F (x,p,t)| \le C(|\na_x E| + |\na_x B| + K(x,t)).
\la{napnaxFbound}
\end{equation}
The properties \eqref{condA}-\eqref{napnaxFbound} are sufficient to obtain global regularity.

\section{On the Glassey-Strauss representation}\la{gs}
Differentiating the Maxwell equations results in the wave equations 
\begin{equation}
\Box E = -\pa_t j - \na_x\rho,
\la{boxe}
\end{equation}
and
\begin{equation}
\Box B = \na_x\times j.
\la{boxb}
\end{equation}
We write 
\begin{equation}
\Box^{-1} g = \int_{|x-y|\le t} \fr{1}{|x-y|} g(y, t-|x-y|)dy.
\la{boxii}
\end{equation}
We consider the the tangential derivatives $T_i$
\begin{equation}
T_i = \pa_i - \omega_i\pa_t
\la{ti}
\end{equation}
with $\omega = (y-x)/|y-x|$, which differentiate in directions parallel to the light cone, 
\begin{equation}
T_i= \fr{\pa}{\pa y_i} (g(y, t-|x-y|)),
\la{tiagain}
\end{equation}
and the derivative
\begin{equation}
V = \pa_t - \omega\cdot \na_y
\la{V}
\end{equation}
which differentiates in the running time $s$ along the light cone,
\begin{equation}
\fr{d}{ds} g(x+ (t-s)\omega, s) = (V g)(x+(t-s)\omega, s).
\la{Vagain}
\end{equation}
We note that 
\begin{equation}
\omega\cdot T + V =0.
\la{degVT}
\end{equation}
Now we note that, if $g = Lh$ where $L$ is a vector field belonging to the linear span of $T_i $ and $V$ and of $h$ is bounded, then $\Box^{-1} g$ is bounded. This is done by integration by parts, using the representation \eqref{boxii} for $Vh$ and $Th$. The linear span can be with variable coefficients depending smoothly on $\omega$.

Glassey and Strauss \cite{glassey1986singularity} represent $E$ and $B$ using the linear wave equations and expressing $\pa_t$ and $\na _y$  as linear combinations of $S$ and $T_i$  where
\begin{equation}
S= \pa_t + v\cdot\na_y
\la{S}
\end{equation}
is the streaming derivative, and where $T_i$ is the tangential derivative given in \eqref{ti}.
The linear combinations are
\begin{equation}
\pa_i = T_i + \fr{\omega_i}{ 1+ \omega \cdot v}\left (S - v\cdot T\right)
\la{paist}
\end{equation}
and
\begin{equation}
\pa_t = \fr{S - v\cdot T}{1 + \omega\cdot v}.
\la{patst}
\end{equation}

This procedure results in two sets of expressions, one coming from the
streaming derivative $S$ and one coming from the tangential derivatives
 $T_{i}$. The overall form is
 \begin{equation}
 \label{GSK}
{\mathbf {K}}(x,t) = ({\mathbf {K}}_{T} + {\mathbf {K}}_{S})(x,t) + O(1)
 \end{equation}
 where $O(1)$ represents a smooth function of $(x,t)$ which depends explicitly on the initial data. For the
 expressions coming from  $S$, we have 
  \begin{equation} 
 \label{KS}  
 \begin{aligned} 
   {\mathbf {K}}_{S}(x,t) &= \int_{|x - y| \leq t}  a_{S}(\omega, v) (Sf)(y, p, t -
 |x - y|) \,dp \frac{dy}{|x - y|} \\
              &= \int_0^{t} (t - s)\,ds \int_{|\omega| = 1} 
              a_{S}(\omega,v) (Sf)(x + (t - s) \omega, p, s) \, dp\,
              \,dS(\omega)  \\
 \end{aligned}  
 \end{equation}  
 where the kernel $a_{S} = a_{S}(\omega,v)$ is an explicit analytic tensor
 valued function satisfying
 \begin{equation}
 \label{eq:aSbd}
   |\nabla_{p} a_{S} | \leq C \lbrack p \rbrack .
 \end{equation}
 The expressions coming from $T$ are
  \begin{equation}
    \begin{aligned}
 \label{KT}
   {\mathbf {K}}_{T}(x,t) &= \int_{|x - y| \leq t}  a_{T}(\omega, v) f(y, p, t -
 |x - y|) \,dp \frac{dy}{|x - y|^2} \\
              &= \int_0^{t} ds \int_{|\omega| = 1} 
              a_{T}(\omega,v) f(x + (t - s) \omega, p, s) \, dp\,
              dS(\omega) 
 \end{aligned}  
 \end{equation}
 where the kernel $a_{T} = a_{T}(\omega,v)$ is an explicit analytic tensor valued function
 satisfying
 \begin{equation}
 \label{eq:aTbd}
   |a_{T}| \leq C\lbrack p \rbrack.
 \end{equation}

For the gradient of the field,  the representation
(\cite{glassey1986singularity} Theorem 4), which is obtained via a
similar procedure, has the form
\begin{equation}
\label{GSgradK}
\nabla_{x}{\mathbf {K}}(x,t) = ((\nabla_{x}{\mathbf {K}})_{TT}+ (\nabla_{x}{\mathbf {K}})_{TS}+
(\nabla_{x}{\mathbf {K}})_{SS})(x,t) + O(1)
\end{equation}
 where  $O(1)$ represents a smooth function of $(x,t)$ which depends explicitly on
the initial data. The terms are
\begin{align}
  (\nabla_{x}{\mathbf {K}})_{TT}(x,t) &= \int_{|x - y| \leq t } a_{TT}(\omega,v)
  f(y,p,t - |x - y|) \, dp \frac{dy}{|x - y|^{3}} \phantom{S}\\ 
  (\nabla_{x}{\mathbf {K}})_{TS}(x,t) &= \int_{|x - y| \leq t } a_{TS}(\omega,v)
  (Sf)(y,p,t - |x - y|) \, dp \frac{dy}{|x - y|^{2}} \\ 
  (\nabla_{x}{\mathbf {K}})_{SS}(x,t) &= \int_{|x - y| \leq t } a_{SS}(\omega,v)
  (S^2f)(y,p,t - |x - y|) \, dp \frac{dy}{|x - y|}. 
\end{align}
Above, the kernels $a_{TT}$, $a_{TS}$ and $a_{SS}$ are explicit tensor
valued analytic functions which satisfy various properties (see \cite{glassey1989large}
Lemma 4). In particular, their derivatives in $y$ and $p$ are bounded by powers of $\lbrack p\rbrack$.

\section{Moment bounds}\la{vm}
In this section we use the radiation reaction force to obtain bounds for moments

\begin{equation}
\label{eq:}
m_{n}(x,t) = \langle \lbrack p \rbrack^{n} \rangle = \int \lbrack p
\rbrack^{n}f(x,p,t) \, dp.
\end{equation}
The charge density $\rho$  corresponds to $m_0(x,t)$ and, as a consequence of the Vlasov equation 
\eqref{rmv}, it obeys the conservation equation
\begin{equation}
\pa_t \rho + \div_x j = 0.
\la{consrho}
\end{equation}
For higher moments, from the Vlasov equation \eqref{rmv}, we have
\begin{equation}
\label{eq:mneq}
\frac{\partial}{\partial t}m_{n} + \div_{x}\langle v \lbrack p
\rbrack^{n} \rangle = n\langle (v \cdot F) \lbrack p \rbrack^{n-1}\rangle,
\end{equation}
where we used
\begin{equation} 
v = \na_p\lbrack p \rbrack
\la{vnalogp}
\end{equation}
and integrated by parts in $\int [p]^n\div_p(Ff)dp$.
 A key element of the proof is provided by the unconditional a priori control of the fluxes $vm_n$ 
  of the moments  $m_n$,
\begin{equation}
vm_n(x,t)= \int f(x,p,t)|v|\lbrack p \rbrack^ndp = \langle |v|\lbrack p \rbrack^n\rangle 
\la{momen}
\end{equation}
in terms of  the initial data.
\begin{thm}\la{momn}
Let $(f, E, B)$ be a smooth solution of the RVM equations on $[0,T]$. 
%Assume that the force $F$  given by \eqref{F} with \eqref{FL} and \eqref{simple} with \eqref{chir} and $M> 2$.
Assume that there exists constant $C_0$ such that 
\begin{equation*}
0\le |p|^3 f_0(x,p)\exp{(A|p|)}  \le C_0
\la{iddecay}
\end{equation*}
 holds for some
\begin{equation*}
A \ge \fr{3 + 2R_0}{(M-2)(R_0)^2}.
\end{equation*} 
Then, for any $n \geq 0$
\begin{equation*}
\sup_{0\le t\le T} \|vm_n(\cdot, t)\|_{L^{\infty}} \le M_n
\la{Mkbound}
\end{equation*}
holds with constants $M_n$ depending explicitly only on $n$, $A$ and $C_0$.
\end{thm}
Theorem \ref{momn} is a corollary of the a priori estimate:
\begin{thm}\la{fdecayb} Let $(f, E, B)$ be a smooth solution of the RVM equations on $[0,T]$. Assume that there exists a constant $C_0$ such that 
\begin{equation}
0\le |p|^3f_0(x,p)\exp{A|p|}  \le C_0
\la{iddecay}
\end{equation}
%Assume that the force $F$  given by \eqref{F} with \eqref{FL} and \eqref{simple} with \eqref{chir} and $M> 2$. 
holds for some
\begin{equation}
A \ge \fr{3 + 2R_0}{(M-2)(R_0)^2}.
\la{MARcond}
\end{equation} 
 Then,
\begin{equation}
0\le f(x,p,t) \le C_0|p|^{-3}\exp({-A|p|})
\la{fpointdecay}
\end{equation}
holds for $t\le T$.
\end{thm} 

\beg{proof}
The path map is defined by the ordinary differential equations
\begin{equation}
\left\{\,\,
\begin{aligned}
\fr{dX}{dt}(a,\pi,t) &= v(P(a,\pi, t)), \quad &&X(a, \pi, 0)= a,\\
{}\\
\fr{dP}{dt}(a,\pi, t) &= F(X(a,\pi, t), P(a, \pi, t), t), \quad &&P(a,\pi, 0)= \pi.
\end{aligned}
\right.
\la{lageq}
\end{equation}
These represent the characteristic curves of the operator
 \begin{equation}
D_t = \pa_t + v\cdot\na_x + F\cdot\na_p.
\la{dt}
\end{equation}

Note that
\begin{equation}
|X(a,\pi,t) -a| < t,
\la{xa}
\end{equation}
because $|v|<1$. This property implies that the decay of $f$ at spatial infinity is controlled for finite time, as long as $F$ is Lipschitz continuous.

We fix a single characteristic 
$X(a,\pi, t)$ and $P(a,\pi, t)$. The equation \eqref{rmv} implies
\begin{equation}
\begin{aligned}
\frac{d}{dt} f(X(a,\pi,t), P(a,\pi,t), t)  &= \\
-(\div_p F(X(a,\pi,t), &\,P(a,\pi,t), t))f(X(a,\pi,t), P(a,\pi,t), t).
\end{aligned}
\la{rmvonchar}
\end{equation}
For the purpose of economy of notation, let us write 
\begin{equation}
r(t) = |P(a,\pi, t)|,
\la{rt}
\end{equation}
for the momentum magnitude,
\begin{equation}
k(t) = K(X(a,\pi,t), t),
\la{kt}
\end{equation}
for the field strength and
\begin{equation}
f(t) =  f(X(a,\pi,t), P(a,\pi,t), t)
\la{ftadhoc}
\end{equation}
for the probability density on characteristics.
These quantities depend on initial data $a$ and $\pi$.

In view of \eqref{divpFsimple},  \eqref{rmvonchar} results in
\begin{equation}
\fr {d}{dt} \log f(t)\le  (3M + |\chi'(r(t))|)k(t).
\la{fineqk}
\end{equation}
For further economy, we suppress that $r, k, f$ are evaluated at $t$.
Using \eqref{Fpineq} we have
\begin{equation}
\begin{aligned}
\fr{dr}{dt}  &\leq - Mkr + (1-\chi(r))k\\
&\leq -k(Mr - (1-\chi(r)))\\&\leq 0
\end{aligned}
\la{dotrineq}
\end{equation} 
where we use the facts that $M\ge 2$ and $(1-\chi(r)) \le |\chi'(r)| r\le 2 r$.
Let us consider the function
\begin{equation} 
\Phi(r) = Ar + \log r^3.
\la{Phir}
\end{equation}
We have that
\begin{equation}
\begin{aligned}
\fr{d}{dt}(\Phi(r) + \log f) &= \Phi'(r) \fr{dr}{dt} + \fr{d}{dt}\log f \\
&\leq -\left(A + \fr{3}{r}\right)(Mkr - (1-\chi)k) + 3Mk + |\chi'|k\\
&\leq -A(Mkr -(1-\chi)k) + \fr{3}{r}(1-\chi)k + |\chi'|k\\
&\leq 0.
\end{aligned}
\la{philogin}
\end{equation}
The last inequality follows because $A$ is large enough \eqref{MARcond}.
Indeed, the supports of $\chi'$ and of $(1-\chi)$ are included in $r\ge R_0$, and 
\begin{equation}
Mkr-(1-\chi)k \ge k(M-2)r\ge k(M-2)R_0
\end{equation} 
there, while $k((1-\chi) 3/r+ |\chi'|) \le k(3/R_0 +2)$.  
We deduce that 
\begin{equation}
\fr{d}{dt} (r^{3} f \exp{Ar}) \le 0.
\la{ftineq}
\end{equation}
We obtained
on each characteristic
\begin{equation}
\begin{aligned}
|P(a,\pi, t)|^{3} f(X(a,&\pi,t), P(a,\pi,t), t)\\&\leq (f_0(a,\pi)|\pi|^{3}\exp{A|\pi |}) \exp{(-A|P(a,\pi, t)|)}.  
\end{aligned}
\la{fineqchar}
\end{equation}
Straightfoward from \eqref{rmvonchar} and $f_0 \geq 0$ is 
\begin{equation}
f(X(a,\pi,t), P(a,\pi,t), t) \geq 0.
\la{nnegchar}
\end{equation}
Reading \eqref{fineqchar} and \eqref{nnegchar} at $x= X(a, \pi, t)$, $p = P(a,\pi,t)$ where $(x,p,t)$ is arbitrary in view of the fact that the flow map is invertible (due to the inverse map theorem of Hadamard, see e.g.   \cite{ruzhansky2015global}) we deduce
  \eqref{fpointdecay}.
\end{proof}

We show that bounds on moment fluxes imply  bounds on  moments which depend logarithmically on gradients of $f$ in either $x$ or $p$. We define 
\begin{equation}
G_1(t) = \sup_{0 \leq s \leq t}\sup_{x,p} |\na_x f(x,p,s)| + 2,
\la{G1}
\end{equation}
and
\begin{equation}
G_2(t) = \sup_{x,p} |\na_p f(x,p,t)| + 2.
\la{G2}
\end{equation}

\begin{thm}\la{mnlogG2}
Let $(f, E, B)$ be a smooth solution of the RVM equations on $[0,T]$. Assume that \eqref{fpointdecay} holds and that  the initial data satisfies $$f_0(x,0) = 0.$$ Then,
\begin{equation*}
m_n(x,t) \le CM_n + C_n\log  G_2(t)
\la{mnlogb}
\end{equation*}
holds for $t\le T$ with a constant $C_n$ depending continuously and explicitly only on $n$ and initial data.
\end{thm}
\beg{proof}
We note first that $f(x,0,t) =0$ holds as long as the solution is smooth (because both $v$ and $F$ vanish at $p=0$). Then, we write
\begin{equation}
\int f(x,p,t)\,dp = \int_{|p|\le R} (f(x,p,t) - f(x,0,t)) \,dp + \int_{|p|\ge 
R}f(x,p,t)\,dp.
\la{splitf1}
\end{equation}
Using \eqref{fpointdecay} 
which implies that $\int_{|p|\ge R}f(x,p,t)dp\le C_0\log{\fr{1}{R}} + \fr{C_0}{A}$, and optimizing in $R$ we obtain
\begin{equation}
\rho (x,t) \le C\log  G_2(t).
\la{rhofty}
\end{equation}
%\be 
%|j(x,t)| \le M_1
%\la{jbound}
%\end{equation}
We have proved the claim for $m_0 = \rho$.
For higher moments, we observe
\begin{equation}
\begin{aligned}
m_n (x,t) &\leq \sqrt 2 (vm_{n}(x,t)) + (\sqrt 2)^n\int_{|p|\le 1} f(x,p,t)dp\\
 &\leq \sqrt 2(vm_n(x,t)) + (\sqrt 2)^n\rho(x,t).
\end{aligned}
\la{mnf1}
\end{equation}
So, the bound on $m_0$ implies bounds on all higher moments, in view of Theorem \ref{momn}.
% Because $|j|\le vm_1$, in view of \eqref{Mkbound}, we obtain \eqref{jbound}.
%Finally, \eqref{mnlogb} follows from \eqref{mnf1}  and \eqref{rhofty}.
\end{proof}

We estimate in terms of $G_1$  the space-time average of $m_n$,
\begin{equation}
\overline{m}_n(x,t) = \frac{1}{ 4 \pi t }\int_{0}^{t}\int_{|\omega| = 1} m_n(x + (t - s)\omega,s) \, dS(\omega) ds.
\end{equation}
Let  us denote the region
\begin{equation}
\Gamma(x,t) = \{ (y,s) : 0 \leq s \leq t, |x - y| \leq t - s\}.
\end{equation}
Fixing $n$ and the vertex $(x,t)$, we consider the quantity
  \begin{equation}
  \label{eq:}
  Q(s) = \int\limits_{|x - y| \leq t - s} m_n(y,s) \, \frac{dy}{|x -
  y|^{2}} 
  \end{equation}
and take the time derivative.  Differentiating, we find
  \begin{equation}
  \label{eq:dGds}
  \frac{dQ}{ds} = - \frac{1}{(t - s)^{2}} \int\limits_{|x - y| = t
  - s} m_{n}(y,s) \,dS(y)
  + \int\limits_{|x - y| \leq t - s} \frac{\partial
m_n}{\partial s}(y,s)  \frac{dy}{|x - y|^{2}}.
  \end{equation}
  Then by the moment evolution law \eqref{eq:mneq} and the property \eqref{Fbound} of $F$
\begin{equation}
\label{eq:consg}
\begin{aligned}
   \int\limits_{|x - y| \leq t - s} &\frac{\partial
m_n}{\partial s}(y,s)  \frac{dy}{|x - y|^{2}} \\&=
\int\limits_{|x - y|
\leq t - s} n \langle (v \cdot F) \lbrack p \rbrack^{n-1} \rangle(y,s) - \div_{y} \langle v\lbrack p\rbrack^n \rangle(y,s) \frac{dy}{|x - y|^{2}}
\\&\leq
\int\limits_{|x - y|
\leq t - s} C_n K(y,s) - \div_{y} \langle v\lbrack p\rbrack^n \rangle(y,s) \frac{dy}{|x - y|^{2}}
\end{aligned}
\end{equation}
where $C_n$ depends only on $n$ and the a priori moment flux bound $M_{n}$ in Theorem \ref{momn}.
Then, integrating by parts
\begin{equation}
\label{eq:divvmineq}
\begin{aligned}
  -&\int\limits_{|x - y|
\leq t - s} \div_{y} \langle v\lbrack p\rbrack^n \rangle(y,s) \frac{dy}{|x - y|^{2}} \\
    =& -\frac{1}{(t - s)^{2}} \int\limits_{|x - y| = t - s} \omega \cdot
\langle v\lbrack p\rbrack^n \rangle(y,s)\, dS(y)  \\
&+ P.V. \int\limits_{|x - y| \leq t - s} \frac{2}{|y -
  x|^{3}}\omega \cdot \langle v\lbrack p\rbrack^n \rangle(y,s) \, dy  \\
        &+ \lim\limits_{\varepsilon \to
0}  \,\langle v\lbrack p\rbrack^n \rangle(x,s) \cdot \int\limits_{|\omega| = 1} \omega \,
dS(\omega).
\end{aligned}
\end{equation}
We observe that on the right hand side of the equality above, the first term is bounded by $M_n$, and the last term  vanishes. 
Then integrating
the equation \eqref{eq:dGds} with respect to $ds$ with \eqref{eq:divvmineq} and \eqref{eq:consg} in hand yields upon dividing by $4\pi t$
 \begin{equation}
\label{eq:mnpv}
\begin{aligned}
  &\overline{m}_n(x,t) \\
  &\leq C_n + C_n \frac{1}{t}\int_0^t \| {\mathbf{K}}(s) \|_{L^{\infty}} ds + \frac{1}{2\pi t} P.V. \int_{\Gamma(x,t)}\frac{1}{|x - y|^3} \omega \cdot \langle v \lbrack p \rbrack^n \rangle \, dy ds.
\end{aligned}
\end{equation}
Here, $t>0$ and $C_n$ depends only on $n$ and the initial data. Indeed, $(4\pi t)^{-1}Q(0)$ is bounded uniformly in $(x,t)$ for smooth data.

\begin{thm}\la{Qnbound} Let $(f, E, B)$ be a smooth solution of  the RVM equations on $[0,T]$. Assume that the initial data $f_0(x,p)$ obeys $f_0(x,0) =0$  and the decay condition \eqref{iddecay}.
 Then
 \begin{equation*}
\overline{m}_n(x,t) \leq  C_n \frac{1}{t }\int_0^t K_\infty(s) ds + C_T(1 + \log  G_1(t))
 \la{qnlogb}
 \end{equation*}
 and
  \begin{equation*}
\overline{m}_n(x,t) \le C_n(1 + \log  G_2(t))
 \la{qnlog2}
 \end{equation*}
holds for $t\le T$ with constant $C_n$ depending continuously and explicitly only on $n$ and  initial data and $C_T$ depending on $n$, initial data and $T$.

\end{thm}
\begin{proof}
The bound for $\overline{m}_n$ in terms of $G_2$ is an immediate consequence of Theorem \ref{mnlogG2}.

To show the bound for $\overline{m}_n$ in terms of $G_1$, we estimate the principal value integral in \eqref{eq:mnpv} as follows. For fixed $s$, we split the spatial integral into the regions $|x - y| \leq \delta$ and $\delta \leq | x - y| \leq t - s$. The value $\delta = \delta(s)$ is chosen below.

The integral on $\delta \leq | x - y| \leq t - s$ is bounded by
\begin{equation}
\left|\int_{\delta \leq |x - y|\leq t - s} \frac{1}{|x - y|^3} \omega \cdot \langle v \lbrack p \rbrack \rangle (s) \, dy\right|\leq C M_n \log \left( \frac{t - s}{\delta}\right).
\end{equation}
For $|x - y| \leq \delta$ and $|p| \geq | x- y|^{-\kappa}$, we evaluate
\begin{equation}
\int_{|p|\ge |x-y|^{-\kappa}} |v|\lbrack p \rbrack^n f(y,p,s)dp \le |x-y|^{k\kappa} vm_{n+k}(y,s)
\la{largepvmn}
\end{equation}
and thus the contribution of this term is bounded,
\begin{equation}
\begin{aligned}
\left|\int_{ |x - y|\leq \delta} \frac{1}{|x - y|^3} \omega \cdot \int_{|p| \geq |x - y|^{-\kappa}} v \lbrack p \rbrack^n f (y,p,s)\,dp \, dy\right|\\ \leq CM_{n+k} \int_{|x - y|\leq \delta} |x - y|^{k\kappa - 3} \,&dy \\
\leq CM_{n + k}&\delta^{k\gamma}.
\end{aligned}
\end{equation}
We are left with the integral for $|x-y| \le \delta$ and $|p|\le |x-y|^{-\kappa}$. Because the unit sphere average 
$\int_{|\omega| =1}(\omega\cdot v) f( x, p, s)dS(\omega)$ vanishes, we have
\begin{equation}
\begin{aligned}
\left|\int_{ |x - y|\leq \delta} \frac{1}{|x - y|^3} \omega \cdot \int_{|p| \leq |x - y|^{-\kappa}} v \lbrack p \rbrack^n f (y,p,s)\,dp \, dy\right|\\ \leq C\sup_{y,p}|\nabla_y f(s)| \int_{|x - y| \leq \delta}\frac{dy}{|x - y|^2}& \int_{|p| \leq |x - y|^{-\kappa}} \lbrack p \rbrack^n \,dp \\
\leq C\sup_{y,p}&|\nabla_y f(s)| \delta^{1 - (n+3)\kappa}.
\end{aligned}
\end{equation}
By choosing $0<\kappa < \frac{1}{n + 3}$ and $\delta = (t - s)/(2 + \sup_{y,p}|\nabla_y f (s)|)$, we find that the time average of the principal value integral is bounded as
\begin{equation}
\begin{aligned}
\left| \frac{1}{2\pi t}\int_0^{t}P.V.\int_{ |x - y|\leq t - s} \frac{1}{|x - y|^3} \omega \cdot \langle v \lbrack p \rbrack \rangle (s) \, dy\,ds\right|
\leq C_T(1 + \log  G_1(t)).
\end{aligned}
\end{equation}
With this estimate and inequality \eqref{eq:mnpv}, we have shown the bound in terms of $G_1$.
\end{proof}

 \section{Electromagnetic field bounds}\la{embounds}
\begin{thm}\la{Kbound} Let $(f, E, B)$ be a smooth solution of the RVM equations on $[0,T]$. Assume that the initial data $f_0(x,p)$  obeys $f_0(x,0) =0$  and the decay condition \eqref{iddecay}.  Let 
\begin{equation*}
K_{\infty}(t) = \sup_{0\le s\le t} \|{\mathbf {K}}(\cdot, s)\|_{L^{\infty}}. 
\la{Kinfty}
\end{equation*}
Then
\begin{equation*}
 K_{\infty}(t)  \le C_1 \left (1 + \min\{ \log  G_1(t), \log  G_2(t)\}\right)
 \la{KftyG1}
 \end{equation*}
holds for $t\le T$ with a constant $C_1$ depending continuously and explicitly only on initial data and $T$.
\end{thm}
  \begin{proof}
    We use the Glassey-Strauss representation \eqref{GSK} for ${\mathbf {K}}$ and
    bound the integrals ${\mathbf {K}}_{S}$ and ${\mathbf {K}}_{T}$. 

    To bound the integral ${\mathbf {K}}_{S}$ with kernel $a_{S}$, we first use 
    the Vlasov equation \eqref{rmv},
    $Sf = -\div_{p}(Ff)$, to
    integrate by parts in $p$, so
    \begin{equation}
    \label{eq:}
    \int a_S Sf \, dp = \int (\nabla_{p}a_{S}) F f \, dp
    \end{equation}
    pointwise in $(y,s)$.
  Then, properties \eqref{Fbound} and \eqref{eq:aSbd} imply
    \begin{equation}
    \label{eq:}
      \begin{aligned}
        \left| \int a_{S} S f \, dp \right| &\leq C \int \lbrack p \rbrack
        |p||{\mathbf{K}}| f\, dp \\
                                            &\leq C M_{2} \| {\mathbf {K}}(s)\|_{L^{\infty}}
      \end{aligned}
    \end{equation}
    because $ |p| \lbrack p \rbrack = |v| \lbrack p \rbrack^2$. 
    Therefore, ${\mathbf {K}}_{S}$ has the bound
    \begin{equation}
    \label{eq:Izerobound}
    \begin{aligned}
      \left| \int_{|x -y| \leq t}\  a_{S} Sf \,
      dp\frac{dy}{|x - y|} \right| &\leq CM_2
      \int_0^{t} (t - s) \| {\mathbf {K}}(s)\|_{L^{\infty}} \, ds \\
                                                             &\leq
                                                             CM_{2}T
                                                             \int_0^{t}\|
                                                             {\mathbf {K}}(s)\|_{L^{\infty}}
                                                             \, ds .
    \end{aligned}
    \end{equation}

    To bound the integral ${\mathbf {K}}_{T}$ with kernel $a_{T}$, we use Theorem
    \ref{Qnbound} because
    $\langle a_T \rangle$ does not generally have a pointwise bound by a  moment
    flux. In
    particular, property \eqref{eq:aTbd} implies
    \begin{equation}
    \label{eq:}
    \begin{aligned}
      \left|   \int_{{|x - y|} \leq t}  a_T f dp\frac{dy}{|x - y|^2}  \right|   \leq C
      T\,\overline{m}_{1}(x,t)
    \end{aligned}
    \end{equation}
    pointwise in $(x,t)$ and then we apply Theorem \ref{Qnbound} for $n =
    1$ using  the bound in terms of either
    $G_1$ or $G_2$. Using the bound in terms of $G_2$,
    we obtain
    \begin{equation}
    \label{eq:Ionebounda}
      \left|   \int_{{|x - y|} \leq t}  a_T f dp \frac{dy}{|x - y|^2}  \right| \leq C_1 M_1 T \log G_2(t).
    \end{equation}
    On the other  hand, from  the bound in terms of $G_1$ we have
    \begin{equation}
    \label{eq:Ioneboundb}
    \left|   \int_{|x - y| \leq t}  a_T f dp \frac{dy}{|x - y|^2}  \right| \leq C_1
    \int_0^{t} \| {\mathbf {K}}(s)\|_{L^{\infty}}\,ds + C_1 \log  G_1(t).
    \end{equation}
    
    To conclude, we apply the estimate \eqref{eq:Izerobound} for ${\mathbf {K}}_{S}$
    with either estimate
    \eqref{eq:Ionebounda}
    or \eqref{eq:Ioneboundb} for ${\mathbf {K}}_{T}$ in the Glassey-Strauss
    representation, and use the Gr{\"o}nwall inequality.
  \end{proof}

\section{Gradient bounds for electromagnetic  fields}\la{gembounds}
 Now that we know the bounds for the moments in Theorem \ref{momn}  and  the uniform $L^{\infty}$ bound on ${\mathbf {K}}$ in Theorem \ref{Kbound}, we can use the Glassey-Strauss representations \eqref{GSgradK} for the spatial gradients of $E$ and $B$ which we denote  by $\na_x {\mathbf {K}}$.

 \begin{thm}\la{thm:naxKbound} Let $(f, E, B)$ be a smooth solution of the RVM equations on $[0,T]$. Assume that the initial data $f_0(x,p)$  obeys $f_0(x,0) = 0$  and the decay condition \eqref{iddecay}. 
 Then
\begin{equation*}
\|\na_x {\mathbf {K}} (\cdot, t)\|_{L^{\infty}} \le C_1\log  G_1(t)\log  G_2(t)
\la{naxKbound}
\end{equation*}
holds for $t\le T$ with a constant $C_1$ depending continuously and explicitly only on initial data and $T$.
\end{thm}
\begin{proof}
  We use the representation  \eqref{GSgradK} for
 the gradient  $\nabla_{x}{\mathbf {K}}$ and bound the integrals $(\nabla_{x}{\mathbf {K}})_{TT}$,
  $(\nabla_{x}{\mathbf {K}})_{TS}$ and $(\nabla_{x}{\mathbf {K}})_{SS}$.

  The simplest term to bound is the integral $(\nabla_{x}{\mathbf {K}})_{TS}$ whose kernel $a_{TS}$ satisfies (\cite{glassey1989large}, Lemma 4)
\begin{equation}
\label{eq:aTSbd}
  |\nabla_{p} a_{TS}| \leq C \lbrack p \rbrack^{4} .
\end{equation}
After using $Sf = -\div_{p}(Ff)$ to integrate by parts, we
  find
  \begin{equation}
  \label{eq:}
     \int a_{TS} Sf \, dp = \int (\nabla_{p}a_{TS}) Ff\,dp.
  \end{equation}
  The properties \eqref{Fbound} and \eqref{eq:aTSbd} then imply
  \begin{equation}
  \label{eq:}
    \begin{aligned}
      \left| \int a_{TS} Sf \, dp \right| &\leq C \int \lbrack p
     \rbrack^{4} |p| |{\mathbf{K}}| f \, dp \\
                                          &\leq C M_{5} \|
                                          {\mathbf {K}}(s)\|_{L^{\infty}}
    \end{aligned}
  \end{equation}
 with the fact $|p| \lbrack p \rbrack^{4} = |v| \lbrack p \rbrack^{5}$.
 Therefore, $(\nabla_{x}{\mathbf {K}})_{TS}$ has the bound
 \begin{equation}
 \label{eq:gradKTS}
   \begin{aligned}
     \left| \int_{|x - y| \leq t} a_{TS} Sf dp \frac{dy}{|x - y|^2}
     \right| &\leq CM_{5} \int_0^{t} \| {\mathbf {K}}(s)\|_{L^{\infty}} \, ds \\
             &\leq C M_{5}T K_{\infty}(t).
   \end{aligned}
 \end{equation}

 In order to bound $(\nabla_{x}{\mathbf {K}})_{SS}$, we first rewrite $S^2 f$
  appealing twice to the Vlasov equation
 $Sf = -\div_{p}(Ff)$.
 Pointwise,
\begin{equation}
\begin{aligned} 
  S(Sf) &= - S(\div_{p}(Ff)) \\
&=\na_x(Ff) : \nabla_{p} v -\div_{p} (S(Ff))  \\
&=\na_x(Ff) : \nabla_{p}v - \div_{p} (fSF) -\div_{p}(FSf)  \\
&=\na_x(Ff) : \nabla_{p}v - \div_{p} (fSF) +  \div_{p} (F\div_{p} (Ff)).
\end{aligned}
\la{s2ex}
\end{equation}
We thus have three terms entering the expression of
$(\nabla_{x}{\mathbf {K}})_{SS}$. For $n = 0,1,2$, the kernel $a_{SS}$  satisfies (\cite{glassey1989large}, Lemma 4)
\begin{equation}
  \label{eq:aSSgradpbd}
  |\nabla_{p}^{n} a_{SS}| \leq C \lbrack p \rbrack^{4} .
\end{equation}

For the last term, integrating by parts in $p$
twice gives
 \begin{equation}
\label{eq:}
\int a_{SS}\, \div_{p} (F\div_{p} (Ff)) \, dp = \int F \cdot \nabla_{p}(F
\cdot \nabla_{p} a_{SS}) f \, dp.
\end{equation}
Then, properties \eqref{Fbound}, \eqref{napFbound} and
\eqref{eq:aSSgradpbd} imply
\begin{equation}
\label{eq:}
\begin{aligned}
  \left| \int  a_{SS} \div_{p}(F \div_{p}(Ff))\,dp \right| &\leq C\int
  |p| \lbrack p \rbrack^{4} ( |\nabla_{p} F| +  |F| )|{\mathbf{K}}|
  f\, dp\\
                                                           &\leq C\int
|v|\lbrack p \rbrack^{5}  |{\mathbf{K}}|^2 f \, dp   \\
                                                           &\leq C M_{5}
                                                           \|
                                                           {\mathbf {K}}(s)\|_{L^{\infty}}^2.
\end{aligned}
\end{equation}
The bound for the last term is therefore
\begin{equation}
\label{eq:aSStermthree}
\begin{aligned}
\left|  \int_{|x - y| \leq t} a_{SS}\, \div_{p}(F \div_{p}(Ff))  dp
\frac{dy}{|x - y|}\right| & \\
\leq C M_{5} \int_0^{t} (t - s)& \|
{\mathbf {K}}(s)\|_{L^{\infty}}^2  ds  \\
                          &\leq CM_{5} (T K_{\infty}(t))^2
\end{aligned}
\end{equation}

For the second term, integration by parts in $p$ yields
 \begin{equation}
\label{eq:}
-\int a_{SS}\,\div_{p}(f SF) \, dp = \int (\nabla_{p} a_{SS}) (SF) f\, dp.
\end{equation}
From the Maxwell equations, 
\begin{equation}
\label{eq:SF}
\begin{aligned}
  SE &=  v \cdot \nabla_{x}E + \nabla_{x} \times B - j,\\ 
  SB &= v\cdot \nabla_{x}B - \nabla_{x} \times E ,
\end{aligned}
\end{equation}
and so from property \eqref{eq:aSSgradpbd}, noting that $S \chi = 0$,
\begin{equation}
\label{eq:}
  \begin{aligned}
    \left| \int a_{SS}\, \div_{p}(fSF)\,dp\right| &\leq C\int  |p| \lbrack p
    \rbrack^{4}  (|SE| + |SB|) f \, dp \\
                                               &\leq C M_{5}  (M_0 + \|
                                               \nabla_{x}{\mathbf {K}}(s)\|_{L^{\infty}}
                                              ).
  \end{aligned}
\end{equation}
The bound for the second term is then
\begin{equation}
\label{eq:aSStermtwo}
\begin{aligned}
  \left| \int_{|x - y| \leq t} a_{SS} \,\div_{p}(f SF) \,dp \frac{dy}{|x -
  y|}  \right|& \\ \leq CM_{5}\int_0^{t}(t - s&)(M_0 +  \|
\nabla_{x}{\mathbf {K}}(s)\|_{L^{\infty}})\,ds   \\
              &\leq CM_5 T \int_0^{t} ( M_0 + \| \nabla_{x}{\mathbf {K}}(s)\|_{L^{\infty}} ) \, ds
\end{aligned}
\end{equation}

For the first term, to integrate by parts the quantity
\begin{equation}
\label{eq:}
\nabla_{x}(Ff) : \nabla_{p}v = \partial_{i}(F_{j} f)\frac{\partial v_{i}}{\partial p_{j}}
\end{equation}
where $\partial_{i} = \partial / \partial x_{i}$,
we recall the decomposition of derivatives
\begin{equation}
\label{eq:}
\partial_{i} = T_{i} + \left(\frac{\omega_{i}}{1 + v \cdot
\omega}\right) ( v \cdot
T - S ). 
\end{equation}
Repeated indices indicate summation. We then write
\begin{equation}
\label{eq:}
a_{SS}  (\nabla_{x} (F f) : \nabla_{p}v ) = A^{ij} T_{i} (F_{j} f)
+ b^{j} S(F_{j}f)
\end{equation}
as the sum of two expressions. 

The latter expression is
\begin{equation}
\label{eq:}
b^{j} S(F_{j}f) = 
  \frac{\partial v_{i}}{\partial p_{j}}
  \left(\frac{\omega_{i}}{1 + v \cdot \omega}\right) a_{SS} S(F_{j}f)
\end{equation}
which becomes
\begin{equation}
\label{eq:}
b^{j} S(F_{j}f) = b^{j}F_{j}Sf + b^{j}f SF_{j}.
\end{equation}
Observe that each term on the right hand side above may be treated in a
similar fashion to terms previously discussed; we use the Vlasov equation to integrate by parts in $p$ and use property
\eqref{Fbound} to deduce
\begin{equation}
\label{eq:bjbdone}
  \left| \int_{|x - y| \leq t} b^{j}F_{j}Sf \, dp \frac{dy}{|x - y|}
  \right| \leq CM_{6} T^2 K_{\infty}(t), 
\end{equation}
and we use properties \eqref{napFbound} and \eqref{eq:SF} to arrive at 
\begin{equation}
\label{eq:bjbdtwo}
\left| \int_{|x - y| \leq t} b^{j} fSF_{j}\, dp \frac{dy}{|x - y|}
\right|   \leq CM_{6} T \int_0^{t} (M_0 + \|
\nabla_{x}{\mathbf {K}}(s)\|_{L^{\infty}})\, ds
\end{equation}

The former expression is
\begin{equation}
\label{eq:}
A^{ij} T_{i}(F_{j}f) = \frac{\partial v_{i}}{\partial p_{j}} \left(
T_{i} +  \frac{\omega_{i}}{1 + v \cdot \omega}v \cdot T\right)(F_{j} f).
\end{equation}
Each  $T_{i}$ is a total  $y$ derivative, and so integrating by parts in
$y$ gives 
 \begin{equation}
\label{eq:}
\int_{|x - y|\leq t} A^{ij} T_{i}(F_{j}f) \, dp \frac{dy}{|x - y|} = -
\int_{|x - y|\leq t} \tilde{A}^{j} (F_{j}f) \, dp \frac{dy}{|x -
y|^2} + O(1)
\end{equation}
where $O(1)$ represents a function of  $(x,t)$ which depends explicitly
on the initial data. On the right hand side is the kernel  $\tilde{A}^{j} = r^2\partial  / \partial
y_{i}(A^{ij} / r)$ where $r = |x - y|$,  which in particular satisfies
$|\tilde{A}^{j}| \leq C \lbrack p \rbrack^{4} $ (see
\cite{glassey1989large} Lemma 4). The estimate for this expression is
then by property \eqref{Fbound}
 \begin{equation}
\label{eq:Aijbd}
\begin{aligned}
\left|\int_{|x - y|\leq t} A^{ij} T_{i}(F_{j}f) \, dp \frac{dy}{|x -
y|} \right| \leq &C\int_{|x - y| \leq t} |v| \lbrack p \rbrack^{5} |{\mathbf{K}}| dp
\frac{dy}{|x - y|^2} + O(1) \\
\leq & C M_{5}\int_0^{t} \| {\mathbf {K}}(s)\|_{L^{\infty}} \, ds + O(1)\\
\leq & C_0 ( 1 + M_{5}TK_{\infty}(t))
\end{aligned}
\end{equation}
where $C_0$ depends only on the initial data. Taking together
\eqref{eq:Aijbd}, \eqref{eq:bjbdone} and \eqref{eq:bjbdtwo} gives us a
bound on the first term entering the expression of $(\nabla_{x}{\mathbf {K}})_{SS}$, while the second and last term have bounds
\eqref{eq:aSStermtwo} and
\eqref{eq:aSStermthree}.

Therefore, $(\nabla_{x} {\mathbf {K}})_{SS}$ has the bound
\begin{equation}
\label{eq:gradKSS}
\left| \int_{|x - y| \leq t} a_{SS} (S^2 f) \,dp
\frac{dy}{|x - y|}\right| \leq C_T \left( 1  + K_{\infty}(t)^2 +
\int_0^{t}\| \nabla_{x}{\mathbf {K}}(s) \|_{L^{\infty}} ds \right) 
\end{equation}
where $C_{T}$ depends only on the initial data and $T$.

To bound $(\nabla_{x}{\mathbf {K}})_{TT}$, we write the integral as
\begin{equation}
\label{eq:}
(\nabla_{x}{\mathbf {K}})_{TT}(x,t) = \int_0^{t}\frac{ds}{t - s}\int_{|\omega| = 1}  a_{TT}(\omega,v) f(x + ( t- s
)\omega, p, s) dp\, dS(\omega)
\end{equation}
We split the integral on the backwards light cone  into two pieces: the base piece on
$0 \leq s \leq t - \delta$, and tip piece on $t - \delta \leq s \leq
t$, where  $\delta$ is chosen below.
The properties of the kernel $a_{TT}$ (\cite{glassey1989large}, Lemma 4), 
\begin{equation}
  |a_{TT}| \leq C \lbrack p \rbrack^{3}
\end{equation}
and
\begin{equation}
  \label{eq:aTTbd}
  \int_{|\omega| = 1} a(v,\omega) dS(\omega) = 0,
\end{equation}
 imply for the base piece
\begin{equation}
\label{eq:}
\left|\int_0^{t- \delta} \frac{ds}{t - s} \int_{|\omega| = 1} 
a_{TT} f \, dp \, dS(\omega) \right| \leq C(M_{4} + \log  G_2(t))\log \left(
\frac{t}{\delta} \right).
\end{equation}
For the tip piece, we first note
\begin{equation}
\label{eq:}
\begin{aligned}
  \left| \int_{|p| \geq (t -s)^{-\kappa}} a_{TT} f \, dp \right| &\leq C
\int_{|p| \geq (t - s)^{-\kappa}} \lbrack p \rbrack^{3} f \, dp  \\
                                                            &\leq C (t -
                                                            s)^{\alpha}
                                                            \int_{|p|
                                                            \geq (t
                                                          -s)^{-\kappa}}
                                                          |p|^{\frac{\alpha}{\kappa}}
                                                          \lbrack p
                                                          \rbrack^{3} f
                                                          \, dp \\
                                                            &\leq
                                                            CM_{n}(t -
                                                            s)^{\alpha}
\end{aligned}
\end{equation}
where $n = \left\lceil  4 + \alpha / \kappa  \right\rceil  $,  and $\alpha, \kappa$ are
numbers chosen freely. We let $\alpha > 0$ so that
\begin{equation}
\label{eq:}
\left| \int^{t}_{t - \delta} \frac{ds}{t - s}\int_{|\omega| = 1} \int_{|p|
\geq (t - s)^{-\kappa}} a_{TT} f \, dp \, dS(\omega)  \right| \leq C
M_{n} \delta^{1 - \alpha }.
\end{equation}
Then, we choose $\kappa < \frac{1}{6}$ such that
\begin{equation}
\label{eq:gradKTTbd}
  \begin{aligned}
\left| \int^{t}_{t - \delta} \frac{ds}{t - s}\int_{|\omega| = 1} \int_{|p|
\leq (t - s)^{-\kappa}} a_{TT} f \, dp \, dS(\omega)  \right| &\\
\leq \sup_{s \leq t} \sup_{x,p} |\nabla_{x} f (x,p,s)| \int_{t -
\delta}^{t} &\int_{| p| \leq (t -s)^{-\kappa}} \lbrack p \rbrack^{3} \, dp \\
\leq  \, \delta^{1 - 6\kappa} \sup_{s \leq t}& \sup_{x,p} |\nabla_{x} f (x,p,s)|.
  \end{aligned}
\end{equation}
With $\alpha = 1$, we choose here $\delta = t (2 + \sup_{s\leq t,x,p} |\nabla_{x} f (x,p,s)|)^{-1 /{(1 -
6\kappa)}}$ in view of the above.

Therefore, we
have the following bound for $(\nabla_{x}{\mathbf {K}})_{TT}$
\begin{equation}
\label{eq:gradKTT}
\left| \int_0^{t}\frac{ds}{t - s} \int_{|\omega| = 1}  a_{TT} f \,
dp \, dS(\omega)  \right| \leq C_{T} \log  G_1(t)
\log  G_2(t).
\end{equation}

Putting together estimates \eqref{eq:gradKTS}, \eqref{eq:gradKSS} and
\eqref{eq:gradKTT}, we obtain
\begin{equation}
\label{eq:}
\begin{aligned}
|\nabla_{x}{\mathbf {K}}(x,t)|\leq C_{T} \left( \log  G_1(t) \log  G_2(t)
+ \int_0^{t} \| \nabla_x {\mathbf {K}}(s)\|_{L^{\infty}}\, ds\right) 
\end{aligned}
\end{equation}
where we chose to bound $K_{\infty}^2$ by the product $C \log  G_1
\log  G_2$ in view of Theorem \ref{Qnbound}. Using
the Gr{\"o}nwall inequality, we conclude the proof.
\end{proof}

\section{Proof of Theorem \ref{main}}\la{gfb}

\begin{thm}\la{naxfbound} Let $(f, E, B)$ be a smooth solution of the RVM equations  on $[0,T]$. Assume that the initial data $f_0(x,p)$  obeys $f_0(x,0) =0$  and the decay condition \eqref{iddecay}. 
 Then
\begin{equation*}
\|\na_x f (\cdot, t)\|_{L^{\infty}}  + \|\lbrack p \rbrack|\na_p f(\cdot,t )|\|_{L^{\infty}}\leq C\exp(C \exp (C t))
\la{naxnaxfbound}
\end{equation*}
holds for $t\le T$ with a constant $C$ depending continuously and explicitly only on initial data.
\end{thm}

\begin{proof}
We consider the quantities
 \begin{equation}
W(t) = \sup_{s\le t} \|\nabla_x f(s)\|_{L^{\infty}} + 3
 \la{Xtdef0}
 \end{equation}
 and
 \begin{equation}
 Z(t) = \sup _{s\le t} \||p| |\nabla_p f(s)| + (1 + |p|) f(s)\|_{L^{\infty}} + 3.
 \la{Ytdef0}
 \end{equation}
 Below we show $W$ and $Z$ obey the certain differential inequalities. We write \eqref{rmv} as
\begin{equation}
D_t f = -(\div_p F)f
\la{Dtf}
\end{equation}
and take derivatives in $x$ and in $p$:
\begin{equation}
D_t (\pa_{x_i} f)  = - (\pa_{x_i} F)\cdot\na_p f  -(\div_p F)(\pa_{x_i}f)   - (\pa_{x_i}(\div_p F)) f
\la{dtnaxf}
\end{equation}
and
\begin{equation}
D_t (\pa_{p_i} f)  =- (\pa_{p_i} v)\cdot\na_x f  -(\div_p F)(\pa_{p_i}f) - (\pa_{p_i} F)\cdot\na_p f   - (\pa_{p_i}(\div_p F)) f
\la{dtnapf}
\end{equation}
We deduce inequalities for quantities
\begin{equation}
w = |\nax f| + 3
\la{g1def}
\end{equation}
and
\begin{equation}
z = (1+|p|)f + |p||\na_p f| + 3.
\la{g2def}
\end{equation}
Using the estimates \eqref{Fbound}, \eqref{napFbound}, \eqref{naxFbound} and
\eqref{napnaxFbound}, we find that
\begin{equation}
D_t w \le C(K w + (K + |\nax {\mathbf {K}}|)z)
\la{dtg1}
\end{equation}
and 
\begin{equation}
D_t z \le C(w + K z)
\la{dtg2}
\end{equation}
from equations \eqref{dtnaxf} and \eqref{dtnapf}.

To see this, first multiply the equation \eqref{dtnaxf} by $\pa_{x_i} f / |\nax f|$ and add in $i$ to obtain,
\begin{equation}
\begin{aligned}
D_t|\nax f| 
 &\le  |\div_p F|| \nax f| + |\na_x F|  |\na_p f | + |\na_x\div_p F| f\\
 &\le  C(K|\nax f|  +  (K + |\na_x {\mathbf {K}}|)(|p||\na_p f| +f))\\
 &\le C(K w + (K + |\nax {\mathbf {K}}|)z).
 \end{aligned} 
\la{gnx}
\end{equation}
This implies \eqref{dtg1}. Then, we multiply \eqref{dtnapf} by $|p|\pa_{p_i} f / |\na_p f|$ and add in $i$ to obtain
 \begin{equation}
\begin{aligned}
|p|D_t|\na_p f| 
 &\le 2|\nax f| + |p| |\div_p F| |\na_p f|  + |p||\na_p\div_p F| f\\
 &\le  2|\nax f| + C(K|p||\na_p f| + K |p| f)\\
 &\le C(w + K z).
 \end{aligned} 
\la{gnp}
\end{equation}
We used $|p| |\na_p v | < 2$, which is immediate from \eqref{vik} and $|v| < 1$. This implies,  with \eqref{Dtf}, the estimate \eqref{dtg2}.
Now we have that
 \begin{equation}
W(t) = \sup_{s\le t} \|w(s)\|_{L^{\infty}},
 \la{Xtdef}
 \end{equation}
 and
 \begin{equation}
Z(t) = \sup _{s\le t} \|z(s)\|_{L^{\infty}}.
 \la{Ytdef}
 \end{equation}
Taking the supremum in time of \eqref{dtg1} and \eqref{dtg2}, we find
\begin{equation}
\begin{aligned}
&\sup_{s \leq t}\| D_t w(s)\|_{L^\infty} \\
&\leq C\left(K_\infty(t) W(t) + \left(K_\infty(t) + \sup_{s \leq t}\|\nax {\mathbf {K}}(s)\|_{L^{\infty}}\right)Z(t)\right)
\end{aligned}
\la{dtX}
\end{equation}
and
\begin{equation}
\sup_{s \leq t}\|D_t z(s) \|_{L^\infty} \le C(K_\infty(t) Z(t) + W(t))
\la{dtY}
\end{equation} 
% In the above expressions, we replace $K$ by 
% $\sup_{s\le t}\|{\mathbf {K}}(s)\|_{L^{\infty}}$ and $|\nax K|$ by $\sup_{s\le t} \|\nax {\mathbf {K}}(s)\|_{L^{\infty}}$.  
%We note that  $\left | \|h(s+ \tau)\|_{L^{\infty}}- \|h(s)\|_{L^{\infty}}\right | \le |\tau| \sup_{\sigma\in [s, s+\tau]}\|D_\sigma h\|_{L^{\infty}}$ holds by evaluating on characteristics.
We use now

\begin{lemma}
\la{diffsup}
    Let $g = g(t)$ be a positive Lipschitz function of $t\in [0, T]$ and let $G(t) =\sup_{s\le t} g(s)$. Then, $G = G(t)$ is Lipschitz and
 \[
 \lim\sup_{h\to 0}\fr{G(t+h)-G(t)}{h} \le \lim\inf_{\varepsilon\to 0}\sup_{s\le t+ \varepsilon}|g'(s)|. 
 \]
\end{lemma}

By Lemma \ref{diffsup}, differentiation under $\sup_{s\le t}$ for Lipschitz functions of time is permissible.  
We have thus
\begin{equation}
\fr{dW}{dt} \le \lim\inf_{\varepsilon\to 0}\sup_{s\le t+\varepsilon} \|D_t w(s)\|_{L^{\infty}}
\la{dxtine}
\end{equation}
and
\begin{equation}
 \fr{dZ}{dt} \le \lim\inf_{\varepsilon\to 0}\sup_{s\le t+\varepsilon} \|D_t z(s)\|_{L^{\infty}}
 \la{dytine}
 \end{equation}
 holds for almost all $t$. Then, using Theorem \ref{Kbound} and Theorem \ref{thm:naxKbound} and the continuity of the upper bounds  allowing to set $\varepsilon=0$, we arrive at the ODE system
 \begin{equation}
\fr{dW}{dt} \le C((\log  W) W + (\log W)(\log Z) Z)
\la{G1eq}
\end{equation}
and
\begin{equation}
\fr{dZ}{dt} \le C((\log Z )Z  + W).
\la{G2eq}
\end{equation}
 We apply Lemma \ref{logsys}:
\begin{lemma}
\label{logsys}
Let $W = W(t)$ and $Z = Z(t)$ be nondecreasing, Lipschitz functions of $t\ge 0$. Let $W(0) =W_0$ and $Z(0) = Z_0$ and suppose
$$
\min{\{ \log W_0, \log Z_0 \}} \geq 1.
$$
 Assume that $W(t)$ and $Z(t)$ obey differential inequalities
\begin{equation*}
\fr{dW}{dt} \le C((\log  W) W + (\log W)(\log Z) Z)
\la{G1eqa}
\end{equation*}
and
\begin{equation*}
\fr{dZ}{dt} \le C((\log Z )Z  + W).
\la{G2eqa}
\end{equation*}
Then the functions $W$ and $Z$ satisfy
\begin{equation*}
W + Z \leq C \exp( C \exp (C t))
\end{equation*}
where $C$ depends only on $W_0$ and $Z_0$.
\end{lemma}
 \begin{rem}
    In contrast, the ODE
$$\fr{dY}{dt}= Y (\log Y)^2 $$
blows up in finite time.
\end{rem} 
\end{proof}
The proof of Theorem \ref{main} is completed now by applying the bounds of Theorem \ref{naxfbound} to the bounds on the EM fields in Theorems \ref{Kbound} and \ref{thm:naxKbound}.

%From \eqref{gnp} we deduce
%\begin{equation}
%\begin{aligned}
%D_t (|p|& |\na_p f|\log(2 +|\na_p f|)) \\
%\le 
%&C[K|p| |\na_p f|\log(2 +|\na_p f|) + \log(2 +|\na_p f |)( |\nax f|  + K |p| f)].
%\end{aligned}
%\la{gnplog}
%\end{equation}

\section*{Appendix A: Checking the nonlinear Glassey-Strauss representation}
Here we derive \eqref{KS} and \eqref{KT} and check the properties \eqref{eq:aSbd} and \eqref{eq:aTbd}.
The expressions for $E$ coming from  $S$, $E_S$ are 
 \cite{glassey1986singularity} p.63,
\begin{equation}
\begin{aligned}
&(E_S)_i\\
 &=  - \int dp\int_0^t \int_{|\omega| =1} \left(\fr{\omega_i + v_i}{1+ (\omega\cdot v)}\right) (Sf)(x- r\omega, p, t-r)rdrdS(\omega)
\end{aligned}
\la{ESGS}
\end{equation}
where $\omega = \widehat{y-x}$. Using  the equation \eqref{rmv}, 
denoting 
\begin{equation}
N(y, p, s) = F(y,p,s)f(y,p,s),
\la{N}
\end{equation}
and integrating by parts in \eqref{ESGS} we obtain
\begin{equation}
\begin{aligned}
&(E_S)_i \\ &= - \int dp\int_0^t \int_{|\omega| =1} \pa_{p_j}\left(\fr{\omega_i + v_i}{1+ (\omega\cdot v)}\right) N_j(x- r\omega, p, t-r)rdrdS(\omega)
\end{aligned}
\la{EsGS}
\end{equation}
which we write as
\begin{equation}
\begin{aligned}
&(E_S)_i\\ &= - \int dp \int_0^t \fr{1}{t-s}\int_{|x-y| = t-s} N(y,s)\cdot \na_p\left (\fr{\omega_i + v_i}{1+ (\omega\cdot v)}\right)dS(y)ds
\end{aligned}
\la{ES}
\end{equation}
The expressions \eqref{ES} for $E_S$  are nonlinear because they employ \eqref{rmv}.  The expression for $E_T$  \cite{glassey1986singularity} p. 63 is
\begin{equation}
\begin{aligned}
&(E_T)_i \\&= -\int dp \int_0^t\fr{1}{(t-s)^2}\int_{|x-y| = t-s} f(y,s) \fr{1}{\lbrack p \rbrack^2}\left(\fr{\omega_i + v_i}{(1+ (\omega\cdot v))^2}\right)dS(y)ds
\end{aligned}
\la{ET}
 \end{equation}
Note that $E_T$ is linear in $f$, because it comes without use of the equation of evolution of $f$.
There are analogous representations  for $B$. The main point here is to verify \eqref{eq:aSbd}
 and \eqref{eq:aTbd}.
We observe that
\begin{equation}
\pa_{p_i}v_k = \fr{1}{\sqrt{1+|p|^2}}\left (\delta_{ik} - v_iv_k\right) = \lbrack p \rbrack^{-1}(\mathbb I - v\otimes v)_{ik}
\la{vik}
\end{equation}
and
\begin{equation}
|v|^2 =  1- \fr{1}{\lbrack p \rbrack^2}.
\la{oneminus}
\end{equation}
We note the following facts. First,
\begin{equation}
\pa_{p_j}\left( \fr{1}{1 + \omega\cdot v}\right) = \fr{v_j}{\lbrack p \rbrack(1 + \omega\cdot v)} - \fr{\omega_j + v_j}{\lbrack p \rbrack(1+ \omega\cdot v)^2}
\la{divsymbo}
\end{equation}
and
\begin{equation}
%\begin{aligned}
\pa_{p_j} \left(\fr{\omega_i + v_i}{1+ (\omega\cdot v)}\right) = \fr{1}{\lbrack p \rbrack}\fr{\left (\delta_{ij} + v_j\omega_i\right)}{1 + (\omega\cdot v)} 
-\fr{1}{\lbrack p \rbrack} \fr{(\omega_i +v_i)(\omega_j + v_j)}{(1 + (\omega\cdot v))^2}.
%\end{aligned}
\la{divsymb}
\end{equation}
These are done by direct calculation, inserting $\omega + v$ terms. The second observation is that
\begin{equation}
\fr{|\omega + v|^2}{ (1+ (\omega\cdot v))^2} = \fr{(1-|v|)^2 + 2|v|\delta}{(1-|v|)^2 + |v|^2\delta^2 + 2(1-|v|)|v|\delta}
\la{symbb}
\end{equation}
where
\begin{equation}
\delta = 1 + \omega \cdot \widehat {p} = 1 + \cos \theta.
\la{delta}
\end{equation}

Multiplying the numerator by $1-|v|$ and using $(1-|v|)^3 \le (1-|v|)^2$ in the numerator we see that the resulting fraction is less than $1$, and therefore, after taking square roots we have,
\begin{equation}
\fr{|\omega + v|}{1 +(\omega\cdot v)} \le \sqrt{2} \lbrack p \rbrack
\la{symbineq}
\end{equation}
where we used 
\begin{equation}
(1-|v|)^{-1} = \lbrack p \rbrack^2(1+|v|)\le 2 \lbrack p \rbrack^2.
\la{vsmall}
\end{equation}
Also, from  $1 + (\omega\cdot v) = 1 + |v|\cos\theta \ge 1-|v|$ and \eqref{vsmall} we have that
\begin{equation}
0\le \fr{1}{1 + \omega\cdot v}\le 2\lbrack p \rbrack^2.
\la{symblow}
\end{equation}
Thus, the second term in \eqref{divsymb} obeys
\begin{equation}
\left |\fr{1}{\lbrack p \rbrack} \fr{(\omega_i +v_i)(\omega_j + v_j)}{(1 + (\omega\cdot v))^2}\right| \le 2\lbrack p \rbrack
\la{secondb}
\end{equation}
and the first term in \eqref{divsymb} is bounded in view of \eqref{symblow} by $4\lbrack p \rbrack$. This implies 
\begin{equation}
\left |\pa_{p_j}\left(\fr{(\omega_i + v_i)}{1 + (\omega\cdot v)}\right)\right | \le 6\lbrack p \rbrack.
\la{divsymbb}
\end{equation}
 Note also that
\begin{equation}
\left |\fr{1}{\lbrack p \rbrack^2}\left(\fr{\omega_i + v_i}{(1+ (\omega\cdot v))^2}\right)\right| \le 2\sqrt{2} \lbrack p \rbrack.
\end{equation}
We verified thus the bounds \eqref{eq:aSbd}  and \eqref{eq:aTbd} in the representation of the electric field.
After use of the equation \eqref{rmv} and integration by parts, the magnetic field representation \cite{glassey1986singularity} p.63, yields
\begin{equation}
B_S  = \int dp \int_0^t \fr{1}{t-s}\int_{|x-y| = t-s} N(y,s)\cdot \na_p\left (\fr{\omega\times v}{1+ (\omega\cdot v)}\right)dS(y)ds
\la{BS}
\end{equation}
From \eqref{divsymbo} and the inequalities \eqref{symbineq} and \eqref{symblow} and   because $ |\omega \times v| \le |\omega+v|$  we have
\begin{equation}
\left |\na_p \left (\fr{\omega\times v}{1+ (\omega\cdot v)}\right)\right | \le 10\lbrack p \rbrack 
\la{sigBb}
\end{equation}
 Finally, the representation of $B_T$ from \cite{glassey1986singularity} is
 \begin{equation}
B_T  = \int dp \int_0^t \fr{1}{(t-s)^2}\int_{|x-y| = t-s} \left (\fr{\omega\times v}{\lbrack p \rbrack^2(1+ (\omega\cdot v))^2}\right)f(y,s) dS(y)ds
\la{BT}
\end{equation}
 and we have
 \begin{equation}
 \left |\fr{\omega\times v}{\lbrack p \rbrack^2(1+ (\omega\cdot v))^2}\right | \le 2\sqrt 2 \lbrack p \rbrack,
 \la{tauBb}
 \end{equation}
 concluding the verification of the inequalities \eqref{eq:aSbd} and \eqref{eq:aTbd}.
%\beg{thebibliography}{99}
%\bibitem{GS1} R. Glassey, W. Strauss, first paper.
%\bibitem{ruzhansky2015global} Michael Ruzhanski and Mitsuri Sugimoto, On global inversion of homogeneous maps,
%\end{thebibliography}

\section*{Appendix B: ODE Lemmas}

We prove here Lemma \ref{diffsup} and
Lemma \ref{logsys}.
\begin{proof}[Proof of Lemma \ref{diffsup}] 
If $G(t) = g(s)$ with $s<t$, then $g(s') \le g(s)$ for all $s\le s' \le t$ (otherwise, $G(t)$ would have been attained at $s'$ not at $s$) and therefore $G(s') = g(s)$ for $s'\in [s,t]$ and the left derivative of $G'(t-0)$  of $G$ at $t$ vanishes. If $g(t) <G(t)$ then $G(s)= G(t)$ for a small interval of $s>t$ and so $G'(t)=0$.

If $g(t) = G(t)$ then for any $\varepsilon>0$ we have 
\begin{equation}
g(s) - g(t) \le (s-t)L_{\varepsilon}
\end{equation}
for all $t<s\le t+\varepsilon$, where $L_{\varepsilon} = \sup_{t\le s \le t+\varepsilon}|g'(s)|$. We take $0<h<\varepsilon$, write $g(s) \le g(t) + hL_{\varepsilon}$ for $s\le t+h$ and take the supremum in $s$ to deduce 
\begin{equation}
G(t+h) \le G(t) + hL_{\varepsilon}.
\end{equation}
Thus $G'(t+0) \le L_{\varepsilon}$.  Because $\varepsilon>0$ is arbitrary, we have $$G'(t+0) \le \lim\inf _{\varepsilon\to 0}L_{\varepsilon}.$$ Finally, if $G(t) = g(t)$ and $g(s)< G(t)$ for all $s<t$ then 
\begin{equation}
G(s_0) \le G(t) + \sup_{s'\le t}|g'(s')| (t-s_0)
\end{equation}
holds by taking supremum of 
\begin{equation}
g(s) \le g(t) + \sup_{s'\le t}|g'(s')|(t-s)
\end{equation}
for $s\le s_0<t$. This concludes the argument.
\end{proof}

%Now we prove Lemma \ref{logsys}.

 \begin{proof}[Proof of Lemma \ref{logsys}]
Consider $\widetilde{Z} = Z\log Z$. The differential inequality for $W$ then reads
\begin{equation}
\fr{dW}{dt}  \le C((\log W) W  + (\log W) \widetilde{Z})
\la{G1Zeq}
\end{equation}
and from the differential inequality for $Z$ we have
\begin{equation}
\fr{d\widetilde{Z}}{dt}  \le C(\log Z +1)(W + \widetilde{Z}).
\la{Zeq}
\end{equation}
Now take $\overline{W} = W + \widetilde{Z}$. Because $\widetilde{Z}\ge Z$ we have $\log Z\le \log \widetilde{Z} \le \log \overline{W}$. We also have
$\log W\le \log \overline{W}$, so we obtain
\begin{equation}
\fr{d \overline{W}}{dt} \le C(\log \overline{W}+1)\overline{W}
\la{Weq}
\end{equation}
and thus $\overline{W}$ is bounded by a double exponential of time.
 \end{proof}  

\noindent{\bf{Data Availability.}} No datasets were generated or analyzed in this paper.

\noindent{\bf{Conflict of Interest.}} The authors declare that they have no conflict of interest.

\noindent{\bf{Acknowledgments.}} The authors thank J. Burby, G. W. Hammett, M. W. Kunz and A. Spitkovsky for helpful discussions. The work of PC was partially supported by NSF grant DMS-2106528 and by
a Simons Collaboration Grant 601960. The work of HG was supported by NSF grant DGE-2039656 and by the Ford Foundation.

\bibliographystyle{plainnat}
\bibliography{local}

\end{document}